\documentclass[11pt]{amsart}
\usepackage{amsfonts}

\newtheorem{theorem}{Theorem}
\theoremstyle{plain}

\newtheorem{corollary}{Corollary}

\newtheorem{lemma}{Lemma}

\newtheorem{remark}{Remark}

\numberwithin{equation}{section}
\input{tcilatex}

\begin{document}
\title[cadlaguity estimates]{On some c\`{a}dl\`{a}guity moment estimates of
stochastic jump processes }
\author{R. Mikulevi\v{c}ius}
\email{mikulvcs@usc.edu}
\address{Department of Mathematics, University of Southern California, Los
Angeles}
\author{Fanhui Xu}
\email{fanhuixu@usc.edu}
\address{Department of Mathematics, University of Southern California, Los
Angeles}
\date{January 2, 2019}

\begin{abstract}
Using X. Fernique's results on the compactness of distributions of c\`{a}dl%
\`{a}g random functions, we derive some c\`{a}dl\`{a}guity moment estimates
for stochastic processes with jumps.
\end{abstract}

\maketitle


\section{Introduction}

As suggested by Kolmogorov, it was proved in \cite{ch} (1956) that if $%
X_{t},~0\leq t\leq 1,$ is a separable real valued process (see \cite{dm})
such that%
\begin{equation}
\mathbf{E}\left[ \left\vert X_{t_{1}}-X_{t_{2}}\right\vert ^{p}\wedge
\left\vert X_{t_{2}}-X_{t_{3}}\right\vert ^{p}\right] \leq C\left\vert
t_{1}-t_{3}\right\vert ^{1+r}  \label{f1}
\end{equation}%
with $r>0,p>0,$ and $C$ independent of $t$, then $X$ has no discontinuities
of the second kind with probability 1. If 
\begin{equation}
\mathbf{E}\left[ \left\vert X_{t_{1}}-X_{t_{2}}\right\vert ^{p}\right] \leq
C\left\vert t_{1}-t_{2}\right\vert ^{1+r}  \label{f0}
\end{equation}%
is assumed instead of of (\ref{f1}), then $X$ paths are H\"{o}lder
continuous (Kolmogorov, 1934). It can be shown (e.g. \cite{i}, \cite{s}),
that under (\ref{f0}), the H\"{o}lder continuity is a consequence of the
well-known Sobolev embedding theorem. In that case, $X$ H\"{o}lder norm
moment estimates can be derived. In this note, we estimate the moments of
time supremum and c\`{a}dl\`{a}g H\"{o}lder coefficient of $X$ in terms of
some integrated time differences of $X$ from which, using assumption (\ref%
{f1}), we can derive the classical claim about the existence of a c\`{a}dl%
\`{a}g modification of $X$. On the other hand, the estimate obtained could
be helpful in the construction of the solutions to SPDEs driven by jump
processes when the method of characteristics with a time reversal is used
(see \cite{dpmt}). Some different type moment estimates were derived in \cite%
{pz} by imposing assumptions on the cumulative distribution function of the
quantities introduced in \cite{c} (see \cite{gs}, Section 4 of Chapter III,
as well).

Our note is organized as follows. In Section 2, we introduce some notation
and state the main claim. Some auxiliary results are presented in Section 3,
and the main theorem is proved in Section 4.

\section{Notation and main result}

Let $E$ be a Polish space with distance $d$ and $D\left( \left[ 0,1\right]
,E\right) $ be the standard space of $E$-valued c\`{a}dl\`{a}g functions on $%
\left[ 0,1\right] .$ For $0\leq s\leq t\leq u\leq 1$, denote $\Delta \left(
f;s,t,u\right) =d(f\left( s\right) ,f\left( t\right) )\wedge d(f\left(
t\right) ,f\left( u\right) ).$ For $\sigma <\tau $, let us introduce the
standard modulus of c\`{a}dl\`{a}guity%
\begin{equation*}
\Delta \left( f;(\sigma ,\tau )\right) =\sup_{\sigma \leq s\leq t\leq u\leq
\tau }d(f\left( s\right) ,f\left( t\right) )\wedge d(f\left( t\right)
,f\left( u\right) ).
\end{equation*}%
For $\mu \in \left( 0,1\right) $, we define $\mu $-H\"{o}lder c\`{a}dl\`{a}g
function space $D^{\mu }\left( [0,1],E\right) $ to be the set of of all $%
f\in D(\left[ 0,1\right] ,E)$ such that%
\begin{equation*}
\left[ f\right] _{\mu }+|f]_{\mu }+[f|_{\mu }<\infty ,
\end{equation*}%
where%
\begin{eqnarray*}
\left[ f\right] _{\mu } &=&\sup_{0\leq s\leq t\leq u\leq 1}\frac{%
d(f(s),f\left( t\right) )\wedge d\left( f\left( t\right) ,f\left( u\right)
\right) }{\left\vert u-s\right\vert ^{\mu }}=\sup_{0\leq s\leq t\leq u\leq 1}%
\frac{\Delta \left( f;s,t,u\right) }{\left\vert u-s\right\vert ^{\mu }}, \\
|f]_{\mu } &=&\sup_{t\in (0,1]}\frac{d\left( f\left( 0\right) ,f\left(
t\right) \right) }{t^{\mu }},~[f|_{\mu }=\sup_{t\in \lbrack 0,1)}\frac{%
d\left( f\left( 1\right) ,f\left( t\right) \right) }{\left\vert
1-t\right\vert ^{\mu }}.
\end{eqnarray*}%
For $\mu \in \left( 0,1\right) ,p>1,f\in D\left( \left[ 0,1\right] ,E\right)
,$ let

\begin{eqnarray*}
\left[ \left[ f\right] \right] _{\mu ,p} &=&\left( \int \int \int_{s<t<u}%
\frac{\left\vert \Delta \left( f;s,t,u\right) \right\vert ^{p}}{\left\vert
u-s\right\vert ^{\mu p+3}}dsdtdu\right) ^{1/p}, \\
||f]]_{\mu ,p} &=&\left( \int \frac{d\left( f\left( 0\right) ,f\left(
t\right) \right) ^{p}}{t^{\mu p+1}}dt\right) ^{1/p},~[[f||_{\mu ,p}=\left(
\int \frac{d\left( f\left( 1\right) ,f\left( t\right) \right) ^{p}}{%
\left\vert 1-t\right\vert ^{\mu p+1}}\right) ^{1/p}.
\end{eqnarray*}

Our main result is the following estimate.

\begin{theorem}
\label{t1}Let $p>1,~\mu \in (0,1)$. There is $C>0$ such that for any $f\in
D^{\mu }([0,1],E)$%
\begin{equation*}
\left[ f\right] _{\mu }+|f]_{\mu }+[f|_{\mu }\leq C\left( [\left[ f\right]
]_{\mu ,p}+||f]]_{\mu ,p}+[[f||_{\mu ,p}\right) .
\end{equation*}%
Moreover, if $E=\mathbf{R}^{d},d\left( x,y\right) =\left\vert x-y\right\vert
,x,y\in \mathbf{R}^{d}$, then%
\begin{equation*}
\sup_{0\leq t\leq 1}\left\vert f\left( t\right) \right\vert \leq C\left(
\left\vert f\right\vert _{L_{p}\left( \left[ 0,1\right] \right) }+[\left[ f%
\right] ]_{\mu ,p}+||f]]_{\mu ,p}+[[f||_{\mu ,p}\right) .
\end{equation*}
\end{theorem}

\begin{remark}
\label{re3}Obviously, for any $E$-valued measurable function $f$ on $\left[
0,1\right] ,$%
\begin{eqnarray*}
&&\left[ \left[ f\right] \right] _{\mu ,p}^{p} \\
&\leq &\int \int \int_{s<t<u}\frac{d(f(s),f\left( t\right) )^{\frac{p}{2}%
}d\left( f\left( t\right) ,f\left( u\right) \right) ^{\frac{p}{2}}}{%
\left\vert u-s\right\vert ^{\mu p+3}}dsdtdu.
\end{eqnarray*}
\end{remark}

\begin{corollary}
\label{c1}Let $\left( \Omega ,\mathcal{F},\mathbf{P}\right) $ be a
probability space and $X:\left[ 0,1\right] \times \Omega \rightarrow E$ be a
measurable function, $p>1,r>0$. Assume that%
\begin{eqnarray*}
\mathbf{E}\left[ \Delta \left( X;s,t,u\right) ^{p}\right] &\leq
&C_{0}\left\vert u-s\right\vert ^{1+r},0\leq s<t<u\leq 1, \\
\mathbf{E}\left[ d\left( X\left( 1\right) ,X\left( t\right) \right) ^{p}%
\right] &\leq &C_{0}\left( 1-t\right) ^{r},0\leq t\leq 1, \\
\mathbf{E}\left[ d\left( X\left( 0\right) ,X\left( t\right) \right) ^{p}%
\right] &\leq &C_{0}t^{r},0\leq t\leq 1,
\end{eqnarray*}%
for some $C_{0}>0$. Then for each $\mu \in \left( 0,1\right) ,\mu <r/p,$
there is a constant $N=N\left( \mu ,r,p\right) $ so that 
\begin{equation}
\mathbf{E}\left( \left[ \left[ X\right] \right] _{\mu ,p}^{p}+||X]]_{\mu
,p}^{p}+[[X||_{\mu ,p}^{p}\right) \leq NC_{0}.  \label{fa1}
\end{equation}%
If $E=\mathbf{R}^{k},d\left( x,y\right) =\left\vert x-y\right\vert ,x,y\in 
\mathbf{R}^{k}$, and \ $X_{\cdot }\in D^{\mu }\left( \left[ 0.1\right]
,E\right) $ a.s. with $\mu \in \left( 0,1\right) ,\mu <r/p,$ then, in
addition,%
\begin{equation*}
\mathbf{E}\left[ \sup_{0\leq t\leq 1}\left\vert X_{t}\right\vert ^{p}\right]
\leq N[C_{0}+\mathbf{E}\int_{0}^{1}\left\vert X_{t}\right\vert ^{p}dt].
\end{equation*}
\end{corollary}

\begin{proof}
Let $r-\mu p>0$. Then%
\begin{equation*}
\mathbf{E}\left( \left[ \left[ X\right] \right] _{\mu ,p}^{p}\right) \leq
C_{0}\int \int_{s<u}\frac{\left\vert u-s\right\vert ^{2+r}}{\left\vert
u-s\right\vert ^{\mu p+3}}dsdu\leq NC_{0}\int_{0}^{1}u^{r-\mu p}du.
\end{equation*}%
Similarly, the other terms can be estimated.

If $E=\mathbf{R}^{k},d\left( x,y\right) =\left\vert x-y\right\vert ,x,y\in 
\mathbf{R}^{k}$, and $X_{\cdot }\in D^{\mu }\left( \left[ 0.1\right]
,E\right) $ a.s. with $\mu \in \left( 0,1\right) ,\mu <r/p,$ then the last
estimate obviously follows by Theorem \ref{t1} and (\ref{fa1}).
\end{proof}

\begin{corollary}
\label{c2}Let $X_{t},t\in \left[ 0,1\right] ,$ be a real valued and
stochastically continuous process. Assume that 
\begin{eqnarray*}
\mathbf{E}\left[ \left\vert X_{t}-X_{s}\right\vert ^{p}\wedge \left\vert
X_{t}-X_{u}\right\vert ^{p}\right] &\leq &C\left\vert u-s\right\vert ^{1+r},
\\
\mathbf{E}\left[ \left\vert X_{0}-X_{t}\right\vert ^{p}\right] &\leq &Ct^{r},
\\
\mathbf{E}\left[ \left\vert X_{1}-X_{t}\right\vert ^{p}\right] &\leq
&C\left\vert 1-t\right\vert ^{r}
\end{eqnarray*}%
for all $0\leq s<t<u\leq 1,$ and some $r>0,p>1$. Then $X$ has a c\`{a}dl\`{a}%
g modification.
\end{corollary}

\begin{proof}
Let 
\begin{equation*}
X_{t}^{n}=X_{\pi _{n}\left( t\right) },\quad t\in \left[ 0,1\right] ,
\end{equation*}%
where $\pi _{n}\left( t\right) =k/2^{n}$ if $k/2^{n}\leq s<(k+1)/2^{n}$, $%
n=1,2,\ldots ,k=0,\ldots ,2^{n}-1$, and $\pi _{n}\left( 1\right) =1-1/2^{n}$%
. Obviously the sequence $X_{\cdot }^{n}\in D^{\beta }\left( \left[ 0,1%
\right] \right) $ a.s. for any $\beta \in \left( 0,1\right) $. Let $\mu \in
\left( 0,1\right) ,\mu p<r$. It is enough to show that 
\begin{equation}
\sup_{n}\mathbf{E}\left[ \left[ \left[ X^{n}\right] \right] _{\mu
,p}^{p}+||X^{n}]]_{\mu ,p}^{p}+[[X^{n}||_{\mu ,p}^{p}\right] <\infty .
\label{f6}
\end{equation}%
Indeed, every $X_{n}$ induces a probability measure $X^{n}\left( \mathbf{P}%
\right) $ on $D\left( \left[ 0,1\right] \right) $. The estimate (\ref{f6})
implies that the sequence of measures $\left\{ X^{n}\left( \mathbf{P}\right)
,n\geq 1\right\} $ is weakly relatively compact (see \cite{gs}, \cite{k}).
Any weak limit of a weakly converging subsequence $X^{n_{k}}$ has c\`{a}dl%
\`{a}g paths with probability 1 and, obviously, the same finite-dimensional
distributions as $X$. Therefore $X$ has a c\`{a}dl\`{a}g modification
according to Lemma 2.24 in \cite{k}.

In order to show (\ref{f6}), we estimate, using assumptions imposed,%
\begin{eqnarray*}
&&\mathbf{E}\left[ \int_{0}^{1}\frac{\left\vert
X_{0}^{n}-X_{t}^{n}\right\vert ^{p}}{t^{\mu p+1}}dt\right] \\
&\leq &\sum_{k=1}^{2^{n}-1}\int_{k/2^{n}}^{\left( k+1\right) /2^{n}}\frac{%
\mathbf{E}\left[ \left\vert X_{0}-X_{k/2^{n}}\right\vert ^{p}\right] }{%
\left( k/2^{n}\right) ^{\mu p+1}}dt\leq C2^{-n}\sum_{k=1}^{2^{n}-1}\left( 
\frac{k}{2^{n}}\right) ^{r-\mu p-1} \\
&=&C\left( 2^{-n}\right) ^{r-\mu p}\sum_{k=1}^{2^{n}-1}k^{r-\mu p-1}\leq
C,n\geq 1.
\end{eqnarray*}%
Similarly,%
\begin{equation*}
\mathbf{E}\left[ \int \frac{\left\vert X_{1}^{n}-X_{t}^{n}\right\vert ^{p}}{%
\left\vert 1-t\right\vert ^{\mu p+1}}dt\right] \leq C,n\geq 1.
\end{equation*}

In the same vein, 
\begin{eqnarray*}
&&\mathbf{E}\int \int \int_{s<t<u}\frac{\left\vert
X_{t}^{n}-X_{s}^{n}\right\vert ^{p}\wedge \left\vert
X_{t}^{n}-X_{u}^{n}\right\vert ^{p}}{\left\vert u-s\right\vert ^{\mu p+3}}%
dsdtdu \\
&\leq
&C\sum_{i=0}^{2^{n}-3}\sum_{j=i+1}^{2^{n}-2}\sum_{k=j+1}^{2^{n}-1}\int_{%
\frac{k}{2^{n}}}^{\frac{k+1}{2^{n}}}\int_{\frac{j}{2^{n}}}^{\frac{j+1}{2^{n}}%
}\int_{\frac{i}{2^{n}}}^{\frac{i+1}{2^{n}}}\frac{\left( \frac{k}{2^{n}}-%
\frac{i}{2^{n}}\right) ^{1+r}}{\left( u-s\right) ^{\mu p+3}}dsdtdu.
\end{eqnarray*}%
Note that $r>\mu p$ and for every set of $\{n,i,j,k\}$, $1/2^{n}\leq \left(
k-i-1\right) /2^{n}\leq u-s\leq \left( k-i+1\right) /2^{n}$. Hence, $\left(
k-i\right) /2^{n}\leq 2\left( u-s\right) $, and 
\begin{eqnarray*}
&&\sum_{i=0}^{2^{n}-3}\sum_{j=i+1}^{2^{n}-2}\sum_{k=j+1}^{2^{n}-1}\int_{%
\frac{k}{2^{n}}}^{\frac{k+1}{2^{n}}}\int_{\frac{j}{2^{n}}}^{\frac{j+1}{2^{n}}%
}\int_{\frac{i}{2^{n}}}^{\frac{i+1}{2^{n}}}\frac{\left( \frac{k}{2^{n}}-%
\frac{i}{2^{n}}\right) ^{1+r}}{\left( u-s\right) ^{\mu p+3}}dsdtdu \\
&\leq &C\int_{0}^{1}\int_{0}^{u}\frac{1}{\left( u-s\right) ^{1+\mu p-r}}%
dsdu\leq C,
\end{eqnarray*}%
which shows that 
\begin{equation*}
\mathbf{E}\int \int \int_{s<t<u}\frac{\left\vert
X_{t}^{n}-X_{s}^{n}\right\vert ^{p}\wedge \left\vert
X_{t}^{n}-X_{u}^{n}\right\vert ^{p}}{\left\vert u-s\right\vert ^{\mu p+3}}%
dsdtdu\leq C,n\geq 1.
\end{equation*}%
Thus (\ref{f6}) follows, and the statement is proved.
\end{proof}

\section{Auxiliary results}

Following \cite{f}, for $0\leq \sigma <\tau \leq 1$, we introduce another
modulus of c\`{a}dl\`{a}guity%
\begin{equation*}
N\left( f;\left( \sigma ,\tau \right) \right) =\inf_{\sigma <\theta \leq
\tau }\sup_{s\in \lbrack \sigma ,\theta ),u\in \left[ \theta ,\tau \right]
}[d\left( f(\sigma ),f(s)\right) \vee d\left( f(u),f\left( \tau \right)
\right) ].
\end{equation*}

Denote for $\eta >0,$%
\begin{equation*}
N\left( f;\eta \right) =\sup_{\sigma <\tau \leq \sigma +\eta }N\left(
f;\left( \sigma ,\tau \right) \right) =\sup_{0<\tau -\sigma \leq \eta
}N\left( f;\left( \sigma ,\tau \right) \right) .
\end{equation*}%
Clearly, $N\left( f;\eta \right) $ is increasing in $\eta $.

\begin{remark}
\label{r1}According to Lemma 1.0 in \cite{f},

(a) For any $\sigma <\tau ,$%
\begin{equation*}
\frac{1}{2}N\left( f;\left( \sigma ,\tau \right) \right) \leq \Delta \left(
f;\left( \sigma ,\tau \right) \right) \leq 2N\left( f;\left( \sigma ,\tau
\right) \right) .
\end{equation*}%
In particular,%
\begin{equation}
\Delta \left( f;\left( 0,1\right) \right) =\sup_{0\leq s\leq t\leq u\leq
1}d(f\left( s\right) ,f\left( t\right) )\wedge d(f\left( t\right) ,f\left(
u\right) )\leq 2N\left( f;\left( 0,1\right) \right) .  \notag
\end{equation}

(b) For $\left( s,t\right) \subseteq \left( \sigma ,\tau \right) ,$%
\begin{equation*}
N\left( f;\left( s,t\right) \right) \leq 2N\left( f;\left( \sigma ,\tau
\right) \right) ,\Delta \left( f;(s,t)\right) \leq \Delta \left( f;(\sigma
,\tau )\right) .
\end{equation*}
\end{remark}

For $\mu \in \left( 0,1\right) $, define 
\begin{equation*}
\left[ f\right] _{\symbol{126}\mu }:=\sup_{\eta >0}\frac{N\left( f;\eta
\right) }{\eta ^{\mu }},f\in D^{\mu }\left( \left[ 0,1\right] \right) .
\end{equation*}

\begin{remark}
\label{cl1}Obviously,%
\begin{eqnarray*}
\left[ f\right] _{\mu } &=&\sup_{0\leq s\leq t\leq u\leq 1}\frac{\Delta
\left( f;s,t,u\right) }{\left\vert u-s\right\vert ^{\mu }}=\sup_{a>0}\sup 
_{\substack{ 0\leq s\leq t\leq u\leq 1,  \\ \left\vert u-s\right\vert \leq a 
}}\frac{\Delta \left( f;s,t,u\right) }{\left\vert u-s\right\vert ^{\mu }} \\
&=&\sup_{a>0}\frac{\sup_{\left\vert u-s\right\vert \leq a,s\leq t\leq
u}\Delta \left( f;s,t,u\right) }{a^{\mu }},
\end{eqnarray*}%
and%
\begin{equation}
\left[ f\right] _{\mu }\leq \sup_{s\leq t\leq u,\left\vert u-s\right\vert
\leq 1/2}\frac{\Delta \left( f;s,t,u\right) }{\left\vert u-s\right\vert
^{\mu }}+2^{\mu }\Delta \left( f;\left( 0,1\right) \right) ;  \label{10}
\end{equation}%
also,%
\begin{equation*}
\left[ f\right] _{\symbol{126}\mu }=\sup_{\eta >0}\frac{N\left( f;\eta
\right) }{\eta ^{\mu }}=\sup_{a>0}\sup_{\eta \leq a}\frac{N\left( f;\eta
\right) }{\eta ^{\mu }}=\sup_{a>0}\frac{\sup_{\eta \leq a}N\left( f;\eta
\right) }{a^{\mu }}.
\end{equation*}
\end{remark}

We show that $\left[ f\right] _{\mu }$ and $\left[ f\right] _{\symbol{126}%
\mu }$ are equivalent.

\begin{lemma}
\label{l1}Let $\mu \in (0,1)$. For any $f\in D^{\mu }\left( [0,1],E\right) $,%
\begin{equation*}
\frac{1}{2}\left[ f\right] _{\symbol{126}\mu }\leq \left[ f\right] _{\mu
}\leq 2\left[ f\right] _{\symbol{126}\mu }.
\end{equation*}
\end{lemma}

\begin{proof}
Since for each $\sigma \leq r\leq \tau ,\tau \leq \sigma +\eta ,$%
\begin{eqnarray*}
\frac{d(f(\sigma ),f\left( r\right) )\wedge d\left( f\left( r\right)
,f\left( \tau \right) \right) }{\eta ^{\mu }} &\leq &\frac{\Delta \left(
f;\left( \sigma ,\tau \right) \right) }{\eta ^{\mu }}\leq \frac{\Delta
\left( f;\left( \sigma ,\tau \right) \right) }{\left( \tau -\sigma \right)
^{\mu }}=\sup_{\sigma \leq s\leq t\leq u\leq \tau }\frac{\Delta \left(
f;s,t,u\right) }{\left\vert \tau -\sigma \right\vert ^{\mu }} \\
&\leq &\sup_{\sigma \leq s\leq t\leq u\leq \tau }\frac{\Delta \left(
f;s,t,u\right) }{\left\vert u-s\right\vert ^{\mu }}\leq \left[ f\right]
_{\mu }
\end{eqnarray*}%
it follows that%
\begin{equation*}
\sup_{\sigma \leq r\leq \tau ,\tau \leq \sigma +\eta }\frac{d(f(\sigma
),f\left( r\right) )\wedge d\left( f\left( r\right) ,f\left( \tau \right)
\right) }{\eta ^{\mu }}\leq \sup_{\sigma \leq \tau ,\tau \leq \sigma +\eta }%
\frac{\Delta \left( f;\left( \sigma ,\tau \right) \right) }{\eta ^{\mu }}%
\leq \left[ f\right] _{\mu },
\end{equation*}%
and, using Remark \ref{cl1},%
\begin{equation}
\left[ f\right] _{\mu }=\sup_{\eta >0}\sup_{\sigma \leq \tau ,\tau \leq
\sigma +\eta }\frac{\Delta \left( f;\left( \sigma ,\tau \right) \right) }{%
\eta ^{\mu }}  \label{fa2}
\end{equation}%
Hence for any $\eta >0,$ by Remark \ref{r1}(a),%
\begin{equation*}
\frac{1}{2}N\left( f;\eta \right) =\frac{1}{2}\sup_{\sigma <\tau \leq \sigma
+\eta }N\left( f;\left( \sigma ,\tau \right) \right) \leq \sup_{\sigma <\tau
\leq \sigma +\eta }\Delta \left( f;\left( \sigma ,\tau \right) \right) \leq
2N\left( f;\eta \right) ,
\end{equation*}%
and by (\ref{fa2}),%
\begin{equation*}
\frac{1}{2}\sup_{\eta >0}\frac{N\left( f;\eta \right) }{\eta ^{\mu }}\leq
\sup_{\eta >0}\sup_{\sigma \leq \tau ,\tau \leq \sigma +\eta }\frac{\Delta
\left( f;\left( \sigma ,\tau \right) \right) }{\eta ^{\mu }}\leq 2\sup_{\eta
>0}\frac{N\left( f;\eta \right) }{\eta ^{\mu }}.
\end{equation*}
\end{proof}

The following key estimate was pointed out in \cite{f}, Lemma 1.2.4, as an
extraction from Theorem 12.5 in \cite{b} (cf. inequality 12.76 in \cite{b}).
For the sake of completeness we provide its proof.

\begin{lemma}
\label{kel}(Lemma 1.2.4 in \cite{f}) For any $f\in D\left( 0.1],E\right) $
and every triplet $0\leq \sigma <t<\tau \leq 1$,%
\begin{equation}
N\left( f;\left( \sigma ,\tau \right) \right) \leq N\left( f;\left( \sigma
,t\right) \right) \vee N\left( f;\left( t,\tau \right) \right) +\Delta
\left( f;\left( \sigma ,t,\tau \right) \right) .  \label{f2}
\end{equation}
\end{lemma}

\begin{proof}
Let $0\leq \sigma <t<\tau \leq 1$. By the definition of $N$, for each $%
\varepsilon \in \left( 0,1\right) $, there exist $\sigma <\theta _{1}\leq
t<\theta _{2}\leq \tau $ such that 
\begin{eqnarray*}
&&d\left( f\left( \sigma \right) ,f\left( s\right) \right) \vee d\left(
f\left( u\right) ,f\left( t\right) \right) \leq N\left( f;\left( \sigma
,t\right) \right) +\varepsilon ,s\in \left[ \sigma ,\theta _{1}\right) ,u\in %
\left[ \theta _{1},t\right] , \\
&&d\left( f\left( t\right) ,f\left( s\right) \right) \vee d\left( f\left(
u\right) ,f\left( \tau \right) \right) \leq N\left( f;\left( t,\tau \right)
\right) +\varepsilon ,s\in \left[ t,\theta _{2}\right) ,u\in \left[ \theta
_{2},\tau \right] .
\end{eqnarray*}%
We split the proof in two cases: $d\left( f\left( \sigma \right) ,f\left(
t\right) \right) \leq d\left( f\left( t\right) ,f\left( \tau \right) \right) 
$ and $d\left( f\left( \sigma \right) ,f\left( t\right) \right) >d\left(
f\left( t\right) ,f\left( \tau \right) \right) $ respectively.

\emph{Case 1}\textbf{: } $d\left( f\left( \sigma \right) ,f\left( t\right)
\right) \leq d\left( f\left( t\right) ,f\left( \tau \right) \right) $, i.e. $%
\Delta \left( f;\sigma ,t,\tau \right) =d\left( f\left( \sigma \right)
,f\left( t\right) \right) $. Obviously, 
\begin{equation}
N\left( f;\left( \sigma ,\tau \right) \right) \leq \sup_{s\in \lbrack \sigma
,\theta _{2})}d\left( f\left( \sigma \right) ,f\left( s\right) \right) \vee
\sup_{u\in \lbrack \theta _{2},\tau ]}d\left( f\left( u\right) ,f\left( \tau
\right) \right) ,  \label{fa3}
\end{equation}%
and 
\begin{equation}
\sup_{u\in \left[ \theta _{2},\tau \right] }d\left( f\left( u\right)
,f\left( \tau \right) \right) \leq N\left( f;\left( t,\tau \right) \right)
+\varepsilon .  \label{fa30}
\end{equation}

Now, 
\begin{equation*}
d\left( f\left( \sigma \right) ,f\left( s\right) \right) \leq N\left(
f;\left( \sigma ,t\right) \right) +\varepsilon ,\text{ if }s\in \left[
\sigma ,\theta _{1}\right) ,
\end{equation*}%
and 
\begin{eqnarray}
d\left( f\left( \sigma \right) ,f\left( s\right) \right) &\leq &d\left(
f\left( s\right) ,f\left( t\right) \right) +d\left( f\left( \sigma \right)
,f\left( t\right) \right)  \notag \\
&\leq &N\left( f;\left( \sigma ,t\right) \right) +\varepsilon +\Delta \left(
f;\sigma ,t,\tau \right) ,\text{ if }s\in \left[ \theta _{1},t\right] , 
\notag \\
d\left( f\left( \sigma \right) ,f\left( s\right) \right) &\leq &d\left(
f\left( s\right) ,f\left( t\right) \right) +d\left( f\left( \sigma \right)
,f\left( t\right) \right)  \notag \\
&\leq &N\left( f;\left( t,\tau \right) \right) +\varepsilon +\Delta \left(
f;\sigma ,t,\tau \right) ,\text{ if }s\in \left( t,\theta _{2}\right) . 
\notag
\end{eqnarray}%
Hence%
\begin{equation}
\sup_{s\in \lbrack \sigma ,\theta _{2})}d\left( f\left( \sigma \right)
,f\left( s\right) \right) \leq N\left( f;\left( \sigma ,t\right) \right)
\vee N\left( f;\left( t,\tau \right) \right) +\Delta \left( f;\sigma ,t,\tau
\right) +\varepsilon .  \label{fa4}
\end{equation}%
Using (\ref{fa3})-(\ref{fa4}), we have 
\begin{equation}
N\left( f;\left( \sigma ,\tau \right) \right) \leq N\left( f;\left( \sigma
,t\right) \right) \vee N\left( f;\left( t,\tau \right) \right) +\Delta
\left( f;\sigma ,t,\tau \right) +\varepsilon .  \label{fa5}
\end{equation}

\emph{Case 2}\textbf{: } $d\left( f\left( \sigma \right) ,f\left( t\right)
\right) >d\left( f\left( t\right) ,f\left( \tau \right) \right) $, i.e. $%
\Delta \left( f;\sigma ,t,\tau \right) =d\left( f\left( t\right) ,f\left(
\tau \right) \right) $. Obviously, 
\begin{equation}
N\left( f;\left( \sigma ,\tau \right) \right) \leq \sup_{s\in \lbrack \sigma
,\theta _{1})}d\left( f\left( \sigma \right) ,f\left( s\right) \right) \vee
\sup_{u\in \lbrack \theta _{1},\tau ]}d\left( f\left( u\right) ,f\left( \tau
\right) \right) ,  \label{fa6}
\end{equation}%
and%
\begin{equation}
\sup_{s\in \lbrack \sigma ,\theta _{1})}d\left( f\left( \sigma \right)
,f\left( s\right) \right) \leq N\left( f;\left( \sigma ,t\right) \right)
+\varepsilon .  \label{fa60}
\end{equation}

Now, 
\begin{eqnarray}
d\left( f\left( u\right) ,f\left( \tau \right) \right) &\leq &N\left(
f;\left( t,\tau \right) \right) +\varepsilon ,\text{ if }u\in \left[ \theta
_{2},\tau \right] ,  \notag \\
d\left( f\left( u\right) ,f\left( \tau \right) \right) &\leq &d\left(
f\left( u\right) ,f\left( t\right) \right) +d\left( f\left( t\right)
,f\left( \tau \right) \right)  \notag \\
&\leq &N\left( f;\left( \sigma ,t\right) \right) +\varepsilon +\Delta \left(
f;\sigma ,t,\tau \right) ,\text{ if }u\in \lbrack \theta _{1},t],  \notag \\
d\left( f\left( u\right) ,f\left( \tau \right) \right) &\leq &N\left(
f;\left( t,\tau \right) \right) +\varepsilon +\Delta \left( f;\sigma ,t,\tau
\right) ,\text{ if }u\in \lbrack t,\theta _{2})  \notag
\end{eqnarray}%
Therefore, again, 
\begin{equation*}
N\left( f;\left( \sigma ,\tau \right) \right) \leq N\left( f;\left( \sigma
,t\right) \right) \vee N\left( f;\left( t,\tau \right) \right) +\Delta
\left( f;\sigma ,t,\tau \right) +\varepsilon .
\end{equation*}

Since $\varepsilon \in \left( 0,1\right) $ is arbitrary, 
\begin{equation*}
N\left( f;\left( \sigma ,\tau \right) \right) \leq N\left( f;\left( \sigma
,t\right) \right) \vee N\left( f;\left( t,\tau \right) \right) +\Delta
\left( f;\sigma ,t,\tau \right) .
\end{equation*}
\end{proof}

For $\mu \in \left( 0,1\right) ,$ define for $f\in D\left( \left[ 0,1\right]
,E\right) ,$%
\begin{equation*}
\left[ f\right] _{\symbol{94}\mu }=\sup_{0<s<u\leq 1}\frac{\Delta \left( f;s,%
\frac{s+u}{2},u\right) }{\left\vert u-s\right\vert ^{\mu }}=\sup_{a>0}\frac{%
\sup_{\left\vert \sigma -\tau \right\vert \leq a}\Delta \left( f;\sigma ,%
\frac{\sigma +\tau }{2},\tau \right) }{a^{\mu }}.
\end{equation*}%
We will need the following equivalence claim.

\begin{lemma}
\label{l2}Let $\mu \in \left( 0,1\right) $. For any $f\in D^{\mu }\left( %
\left[ 0,1\right] ,E\right) ,$%
\begin{equation*}
\left[ f\right] _{\symbol{94}\mu }\leq \left[ f\right] _{\mu }\leq 2\left[ f%
\right] _{\symbol{126}\mu }\leq \frac{2}{1-2^{-\mu }}\left[ f\right] _{%
\symbol{94}\mu }.
\end{equation*}
\end{lemma}

\begin{proof}
Let $K=\left[ f\right] _{\symbol{126}\mu }$. According to Remark \ref{cl1},
for any $a>0,$ 
\begin{equation*}
N(f;a/2)\leq Ka^{\mu }2^{-\mu }.
\end{equation*}%
By (\ref{f2}), for every $\sigma <\tau \leq 1$ such that $\left\vert \tau
-\sigma \right\vert \leq a,$ we have, by Lemma \ref{kel},%
\begin{eqnarray*}
&&N\left( f;\left( \sigma ,\tau \right) \right) \\
&\leq &N\left( f;\left( \sigma ,\frac{\sigma +\tau }{2}\right) \right) \vee
N\left( f;\left( \frac{\sigma +\tau }{2},\tau \right) \right) +\Delta \left(
f;\sigma ,\frac{\sigma +\tau }{2},\tau \right) \\
&\leq &N\left( f;a/2\right) +\Delta \left( f;\sigma ,\frac{\sigma +\tau }{2}%
,\tau \right) \\
&\leq &Ka^{\mu }2^{-\mu }+\Delta \left( f;\sigma ,\frac{\sigma +\tau }{2}%
,\tau \right) .
\end{eqnarray*}%
Hence 
\begin{equation*}
N\left( f;a\right) \leq Ka^{\mu }2^{-\mu }+\sup_{0<\tau -\sigma \leq
a}\Delta \left( f;\left( \sigma ,\frac{\sigma +\tau }{2},\tau \right) \right)
\end{equation*}%
and%
\begin{equation*}
\frac{N\left( f;a\right) }{a^{\mu }}\leq 2^{-\mu }K+\frac{\sup_{0<\tau
-\sigma \leq a}\Delta \left( f;\left( \sigma ,\frac{\sigma +\tau }{2},\tau
\right) \right) }{a^{\mu }}
\end{equation*}%
Taking $\sup $ in $a>0$ on both sides, we see that 
\begin{equation*}
K\leq 2^{-\mu }K+\sup_{a>0}\frac{\sup_{0<\tau -\sigma \leq a}\Delta \left(
f;\sigma ,\frac{\sigma +\tau }{2},\tau \right) }{a^{\mu }}
\end{equation*}%
or%
\begin{eqnarray*}
\left[ f\right] _{\symbol{126}\mu } &=&K\leq \frac{1}{1-2^{-\mu }}\sup_{a>0}%
\frac{\sup_{0<\tau -\sigma \leq a}\Delta \left( f;\sigma ,\frac{\sigma +\tau 
}{2},\tau \right) }{a^{\mu }} \\
&=&\frac{1}{1-2^{-\mu }}\left[ f\right] _{\symbol{94}\mu }.
\end{eqnarray*}%
The statement follows by Lemma \ref{l1}.
\end{proof}

\section{Proof of Theorem \protect\ref{t1}}

First we show that there is $C=C\left( \mu ,p\right) $ so that with any $%
\delta \in \left( 0,1\right) ,$%
\begin{equation}
d\left( f\left( 0\right) ,f\left( t\right) \right) \leq Ct^{\mu }\left( 
\left[ f\right] _{\mu }\delta ^{\mu }+\delta ^{-1/p}||f]]_{\mu ,p}\right)
,t\leq 3/4.  \label{fo1}
\end{equation}

Assume $0<t\leq 3/4$. Taking $\varepsilon <\frac{1}{4}t$, we have for $\tau
^{\prime }\in \left( t-\varepsilon ,t\right) ,\tau ^{\prime \prime }\in
\left( t,t+\varepsilon \right) $, 
\begin{equation*}
d\left( f\left( 0\right) ,f\left( t\right) \right) \leq \Delta \left( f;\tau
^{\prime },t,\tau ^{\prime \prime }\right) +d\left( f\left( 0\right)
,f\left( \tau ^{\prime }\right) \right) +d\left( f\left( 0\right) ,f\left(
\tau ^{\prime \prime }\right) \right) .
\end{equation*}%
Integrating with respect to $\tau ^{\prime },\tau ^{\prime \prime }$ over $%
\bar{Q}=\left[ t-\varepsilon ,t\right] \times \lbrack t,t+\varepsilon ],$ 
\begin{eqnarray}
&&\varepsilon ^{2}d\left( f\left( 0\right) ,f\left( t\right) \right)
\label{fo20} \\
&\leq &\varepsilon \int_{t-\varepsilon }^{t}d\left( f\left( 0\right)
,f\left( \tau ^{\prime }\right) \right) d\tau ^{\prime }+\varepsilon
\int_{t}^{t+\varepsilon }d\left( f\left( 0\right) ,f\left( \tau ^{\prime
\prime }\right) \right) d\tau ^{\prime \prime }  \notag \\
&&+\int \int_{\bar{Q}}\Delta \left( f;\tau ^{\prime },t,\tau ^{\prime \prime
}\right) d\tau ^{\prime }d\tau ^{\prime \prime }=A_{1}+A_{2}+B.  \notag
\end{eqnarray}%
Now%
\begin{equation*}
B\leq \left[ f\right] _{\mu }\int_{t-\varepsilon
}^{t}\int_{t}^{t+\varepsilon }\left( \tau ^{\prime \prime }-\tau ^{\prime
}\right) ^{\mu }d\tau ^{\prime \prime }d\tau ^{\prime }\leq C\left[ f\right]
_{\mu }\varepsilon ^{2+\mu }.
\end{equation*}%
Taking $\varepsilon =\frac{1}{4}\delta t$ with any $\delta \in \left(
0,1\right) ,$ we have%
\begin{equation}
\varepsilon ^{-2}B\leq C\left[ f\right] _{\mu }\delta ^{\mu }t^{\mu }.
\label{fo3}
\end{equation}%
By H\"{o}lder inequality, $\kappa =\mu +1/p,1/p+1/q=1,$%
\begin{eqnarray*}
A_{1} &=&\varepsilon \int_{t-\varepsilon }^{t}\frac{d\left( f\left( 0\right)
,f\left( \tau ^{\prime }\right) \right) }{\tau ^{\prime }}\tau ^{\prime
}d\tau ^{\prime } \\
&\leq &\varepsilon \left( \int_{t-\varepsilon }^{t}\frac{d\left( f\left(
0\right) ,f\left( \tau ^{\prime }\right) \right) ^{p}}{\tau ^{\prime \kappa
p}}d\tau ^{\prime }\right) ^{1/p}\left( \int_{t-\varepsilon }^{t}\tau
^{\prime \kappa q}d\tau ^{\prime }\right) ^{1/q} \\
&=&C\varepsilon \left[ t^{1+\kappa q}-\left( t-\varepsilon \right)
^{1+\kappa q}\right] ^{1/q}||f]]_{\mu ,p}\leq Ct^{\kappa }\varepsilon
^{1+1/q}||f]]_{\mu ,p}.
\end{eqnarray*}%
Taking $\varepsilon =\frac{1}{4}\delta t$, with any $\delta \in \left(
0,1\right) ,$%
\begin{equation}
\varepsilon ^{-2}A_{1}\leq Ct^{\kappa }\varepsilon ^{1/q-1}A_{1}=Ct^{\kappa
}\varepsilon ^{-1/p}A_{1}=Ct^{\mu }\delta ^{-1/p}||f]]_{\mu ,p}  \label{fo4}
\end{equation}%
and, the same way,%
\begin{equation}
\varepsilon ^{-2}A_{2}\leq Ct^{\mu }\delta ^{-1/p}||f]]_{\mu ,p}.
\label{fo5}
\end{equation}%
The inequality (\ref{fo1}) follows from (\ref{fo20})-(\ref{fo5}).

Similarly, with obvious changes, we prove that%
\begin{equation}
d\left( f\left( 1\right) ,f\left( t\right) \right) \leq C\left( 1-t\right)
^{\mu }\left( \left[ f\right] _{\mu }\delta ^{\mu }+\delta ^{-1/p}[[f||_{\mu
,p}\right) ,t\geq 1/4,  \label{fo2}
\end{equation}%
for some $C=C\left( \mu ,p\right) $ with any $\delta \in \left( 0,1\right) .$

Finally, (\ref{fo1}), (\ref{fo2}) imply that there is $C=C\left( \mu
,p\right) $ so that with any $\delta \in \left( 0,1\right) ,$ 
\begin{eqnarray}
&&d\left( f\left( 1\right) ,f\left( 0\right) \right)  \label{f5} \\
&\leq &d\left( f\left( \frac{1}{2}\right) ,f\left( 0\right) \right) +d\left(
f\left( 1\right) ,f\left( \frac{1}{2}\right) \right)  \notag \\
&\leq &C\left( \frac{1}{2}\right) ^{\mu }\left( \left[ f\right] _{\mu
}\delta ^{\mu }+\delta ^{-1/p}[[f||_{\mu ,p}\right) +C\left( \frac{1}{2}%
\right) ^{\mu }\delta ^{-1/p}||f]]_{\mu ,p}.  \notag
\end{eqnarray}%
Now,%
\begin{eqnarray*}
d\left( f\left( 1\right) ,f\left( t\right) \right) &\leq &d\left( f\left(
1\right) ,f\left( 0\right) \right) +d\left( f\left( 0\right) ,f\left(
t\right) \right) \text{ if }t\in \left( 0,1/4\right) , \\
d\left( f\left( 0\right) ,f\left( t\right) \right) &\leq &d\left( f\left(
1\right) ,f\left( 0\right) \right) +d\left( f\left( 1\right) ,f\left(
t\right) \right) \text{ if }t\in \left( 3/4,1\right) .
\end{eqnarray*}%
Hence, by (\ref{fo1}), (\ref{fo2}) and (\ref{f5}), there is $C=C\left( \mu
,p\right) $ so that for all $\delta \in \left( 0,1\right) ,t\in \left(
0,1\right) ,$%
\begin{equation}
\lbrack f|_{\mu }+|f]_{\mu }\leq C\left( \left[ f\right] _{\mu }\delta ^{\mu
}+\delta ^{-1/p}[[f||_{\mu ,p}+\delta ^{-1/p}||f]]_{\mu ,p}\right) .
\label{f51}
\end{equation}

Now we estimate $\Delta \left( f;s,t,u\right) $ with $0\leq s<t<u$ $\leq 1$
and $t=\frac{s+u}{2}$.

(i) Assume $\left\vert s\right\vert >\frac{1}{4}\left\vert u-s\right\vert $
and $\left\vert 1-u\right\vert >\frac{1}{4}\left\vert u-s\right\vert .$

Let $\varepsilon <\frac{1}{4}\left( u-s\right) ,s^{\prime }\in \left(
s-\varepsilon ,s\right) ,s^{\prime \prime }\in \left( s,s+\varepsilon
\right) ,t^{\prime }\in \left( t-\varepsilon ,t\right) ,t^{\prime \prime
}\in \left( t,t+\varepsilon \right) ,u^{\prime }\in \left( u-\varepsilon
,u\right) ,u^{\prime \prime }\in \left( u,u+\varepsilon \right) ,$ and%
\begin{eqnarray*}
A &=&\Delta \left( f;s^{\prime },s,s^{\prime \prime }\right) +\Delta \left(
f;t^{\prime },t,t^{\prime \prime }\right) +\Delta \left( f;u^{\prime
},u,u^{\prime \prime }\right) , \\
B &=&\Delta \left( f;s^{\prime },t^{\prime },u^{\prime \prime }\right)
+\Delta \left( f;s^{\prime },t^{\prime },u^{\prime }\right) +\Delta \left(
f;s^{\prime },t^{\prime \prime },u^{\prime \prime }\right) +\Delta \left(
f;s^{\prime },t^{\prime \prime },u^{\prime }\right) \\
&&+\Delta \left( f;s^{\prime \prime },t^{\prime },u^{\prime \prime }\right)
+\Delta \left( f;s^{\prime \prime },t^{\prime },u^{\prime }\right) +\Delta
\left( f;s^{\prime \prime },t^{\prime \prime },u^{\prime \prime }\right)
+\Delta \left( f;s^{\prime \prime },t^{\prime \prime },u^{\prime }\right) .
\end{eqnarray*}%
Let $Q=\left( s-\varepsilon ,s\right) \times \left( s,s+\varepsilon \right)
\times \left( t-\varepsilon ,t\right) \times \left( t,t+\varepsilon \right)
\times \left( u-\varepsilon ,u\right) \times \left( u,u+\varepsilon \right) $%
. Then $\left\vert Q\right\vert =\varepsilon ^{6},$ and%
\begin{equation}
\Delta \left( f;s,t,u\right) \leq A+B.  \label{f3}
\end{equation}%
Let%
\begin{eqnarray*}
\tilde{A} &=&\varepsilon ^{-4}\int_{Q}A=\int_{s-\varepsilon
}^{s}\int_{s}^{s+\varepsilon }\Delta \left( f;s^{\prime },s,s^{\prime \prime
}\right) ds^{\prime \prime }ds^{\prime }+\int_{t-\varepsilon
}^{t}\int_{t}^{t+\varepsilon }\Delta \left( f;t^{\prime },t,t^{\prime \prime
}\right) dt^{\prime \prime }dt^{\prime } \\
&&+\int_{u-\varepsilon }^{u}\int_{u}^{u+\varepsilon }\Delta \left(
f;u^{\prime },u,u^{\prime \prime }\right) du^{\prime \prime }du^{\prime },
\end{eqnarray*}%
and%
\begin{eqnarray*}
&&\tilde{B} \\
&=&\varepsilon ^{-3}\int_{Q}B=\int_{s-\varepsilon }^{s}\int_{t-\varepsilon
}^{t}\int_{u}^{u+\varepsilon }\Delta \left( f;s^{\prime },t^{\prime
},u^{\prime \prime }\right) du^{\prime \prime }dt^{\prime }ds^{\prime } \\
&&+\int_{s-\varepsilon }^{s}\int_{t-\varepsilon }^{t}\int_{u-\varepsilon
}^{u}\Delta \left( f;s^{\prime },t^{\prime },u^{\prime }\right) du^{\prime
}dt^{\prime }ds^{\prime } \\
&&+\int_{s-\varepsilon }^{s}\int_{t}^{t+\varepsilon }\int_{u}^{u+\varepsilon
}\Delta \left( f;s^{\prime },t^{\prime \prime },u^{\prime \prime }\right)
du^{\prime \prime }dt^{\prime \prime }ds^{\prime }+\int_{s-\varepsilon
}^{s}\int_{t}^{t+\varepsilon }\int_{u-\varepsilon }^{u}\Delta \left(
f;s^{\prime },t^{\prime \prime },u^{\prime }\right) du^{\prime }dt^{\prime
\prime }ds^{\prime } \\
&&+\int_{s}^{s+\varepsilon }\int_{t-\varepsilon }^{t}\int_{u}^{u+\varepsilon
}\Delta \left( f;s^{\prime \prime },t^{\prime },u^{\prime \prime }\right)
du^{\prime \prime }dt^{\prime }ds^{\prime \prime }+\int_{s}^{s+\varepsilon
}\int_{t-\varepsilon }^{t}\int_{u-\varepsilon }^{u}\Delta \left( f;s^{\prime
\prime },t^{\prime },u^{\prime }\right) du^{\prime }dt^{\prime }ds^{\prime
\prime } \\
&&+\int_{s}^{s+\varepsilon }\int_{t}^{t+\varepsilon }\int_{u}^{u+\varepsilon
}\Delta \left( f;s^{\prime \prime },t^{\prime \prime },u^{\prime \prime
}\right) du^{\prime \prime }dt^{\prime }ds^{\prime \prime } \\
&&+\int_{s}^{s+\varepsilon }\int_{t}^{t+\varepsilon }\int_{u-\varepsilon
}^{u}\Delta \left( f;s^{\prime \prime },t^{\prime \prime },u^{\prime
}\right) du^{\prime }dt^{\prime \prime }ds^{\prime \prime } \\
&=&\tilde{B}_{1}+\ldots +\tilde{B}_{8}.
\end{eqnarray*}%
Integrating (\ref{f3}) over $Q,$%
\begin{equation}
\Delta \left( f;s,t,u\right) \leq \varepsilon ^{-2}\tilde{A}+\varepsilon
^{-3}\tilde{B}.  \label{f4}
\end{equation}%
Now,%
\begin{eqnarray*}
&&\int_{s-\varepsilon }^{s}\int_{s}^{s+\varepsilon }\Delta \left(
f;s^{\prime },s,s^{\prime \prime }\right) ds^{\prime }ds^{\prime \prime } \\
&\leq &\left[ f\right] _{\mu }\int_{s-\varepsilon
}^{s}\int_{s}^{s+\varepsilon }(s^{\prime \prime }-s^{\prime })^{\mu
}ds^{\prime }ds^{\prime \prime }\leq C\left[ f\right] _{\mu }\varepsilon
^{2+\mu }.
\end{eqnarray*}%
Similarly we estimate the other two terms in $\tilde{A}$ and see that%
\begin{equation}
\varepsilon ^{-2}\tilde{A}\leq C_{1}\left[ f\right] _{\mu }\varepsilon ^{\mu
}.  \label{f7}
\end{equation}

Using H\"{o}lder inequality, $1/q+1/p=1,p>1,$ with $\kappa =\mu +3/p,$ 
\begin{eqnarray*}
&&\varepsilon ^{-3}\tilde{B}_{1} \\
&=&\varepsilon ^{-3}\int_{s}^{s+\varepsilon }\int_{t-\varepsilon
}^{t}\int_{u}^{u+\varepsilon }\frac{d\left( f\left( s^{\prime }\right)
,f\left( t^{\prime }\right) \right) \wedge d\left( f\left( u^{\prime \prime
}\right) ,f\left( t^{\prime }\right) \right) }{\left\vert s^{\prime \prime
}-u^{\prime \prime }\right\vert ^{\kappa }}\left\vert s^{\prime \prime
}-u^{\prime \prime }\right\vert ^{\kappa }ds^{\prime }dt^{\prime }du^{\prime
\prime } \\
&\leq &\varepsilon ^{-3}\left( \int_{s-\varepsilon }^{s}\int_{t-\varepsilon
}^{t}\int_{u}^{u+\varepsilon }(u^{\prime \prime }-s^{\prime })^{\kappa
q}ds^{\prime }dt^{\prime }du^{\prime \prime }\right) ^{1/q} \\
&&\times \left( \int_{s-\varepsilon }^{s}\int_{t-\varepsilon
}^{t}\int_{u}^{u+\varepsilon }\frac{\Delta \left( f;s^{\prime },t^{\prime
},u^{\prime \prime }\right) ^{p}}{\left\vert u^{\prime \prime }-s^{\prime
}\right\vert ^{\kappa }}ds^{\prime }dt^{\prime }du^{\prime \prime }\right)
^{1/p} \\
&\leq &\varepsilon ^{-3}\left( \int_{s}^{s+\varepsilon }\int_{t-\varepsilon
}^{t}\int_{u}^{u+\varepsilon }\left\vert s^{\prime \prime }-u^{\prime \prime
}\right\vert ^{q\kappa }ds^{\prime \prime }dt^{\prime }du^{\prime \prime
}\right) ^{1/q}\left[ [f\right] ]_{\mu ,p}\mathbf{.}
\end{eqnarray*}%
Since 
\begin{eqnarray*}
&&\int_{s-\varepsilon }^{s}\int_{t-\varepsilon }^{t}\int_{u}^{u+\varepsilon
}(u^{\prime \prime }-s^{\prime })^{\kappa q}ds^{\prime }dt^{\prime
}du^{\prime \prime } \\
&\leq &C\varepsilon \left( u-s\right) ^{\kappa q+2}\frac{\varepsilon ^{2}}{%
\left( u-s\right) ^{2}}=C\varepsilon ^{3}\left( u-s\right) ^{\kappa q},
\end{eqnarray*}%
we have%
\begin{equation*}
\varepsilon ^{-3}\tilde{B}_{1}\leq C\varepsilon ^{-3+\frac{3}{q}}\left(
u-s\right) ^{\kappa }\left[ \left[ f\right] \right] _{\mu ,p}.
\end{equation*}%
Similarly estimating the other terms in $\tilde{B}$, we get%
\begin{eqnarray*}
\varepsilon ^{-3}\tilde{B} &\leq &C\varepsilon ^{-3+3/q}\left( u-s\right)
^{\kappa }\left[ [f\right] ]_{\mu ,p}=C\varepsilon ^{-3/p}\left( u-s\right)
^{\kappa }\left[ [f\right] ]_{\mu ,p}. \\
&=&C\delta ^{-3/p}\left( u-s\right) ^{\kappa -3/p}\left[ [f\right] ]_{\mu
,p}=C\delta ^{-3/p}\left( u-s\right) ^{\mu }\left[ [f\right] ]_{\mu ,p}%
\mathbf{.}
\end{eqnarray*}%
Hence by (\ref{f4}) and (\ref{f7}), for some $C=C\left( \mu ,p\right) ,$%
\begin{equation}
\Delta \left( f;s,t,u\right) \leq C_{1}\left[ f\right] _{\mu }\varepsilon
^{\mu }+C\varepsilon ^{-3/p}\left( u-s\right) ^{\kappa }\left[ [f\right]
]_{\mu ,p}  \label{f70}
\end{equation}%
if $\left\vert s\right\vert >\frac{1}{4}\left\vert u-s\right\vert $ and$%
\left\vert 1-u\right\vert >\frac{1}{4}\left\vert u-s\right\vert $. Taking $%
\varepsilon =\frac{1}{4}\delta \left( u-s\right) $\ with any $\delta \in
\left( 0,1\right) ,$ we have%
\begin{equation}
\Delta \left( f;s,t,u\right) \leq C\left( u-s\right) ^{\mu }\left( \left[ f%
\right] _{\mu }\delta ^{\mu }+\delta ^{-\frac{3}{p}}\left[ [f\right] ]_{\mu
,p}\right)  \label{f71}
\end{equation}%
for some $C=C\left( \mu ,p\right) $ if $\left\vert s\right\vert >\frac{1}{4}%
\left\vert u-s\right\vert $ and$\left\vert 1-u\right\vert >\frac{1}{4}%
\left\vert u-s\right\vert $.

(ii) Assume $\left\vert s\right\vert \leq \frac{1}{4}\left\vert
u-s\right\vert $ or $\left\vert 1-u\right\vert \leq \frac{1}{4}\left\vert
u-s\right\vert .$

If $s\leq \frac{1}{4}\left\vert u-s\right\vert $ (recall $t=\frac{s+u}{2}$),
then $s\leq 1/4$ and%
\begin{equation*}
t=s+\frac{u-s}{2}\leq \frac{3}{4}\left( u-s\right) \leq \frac{3}{4}.
\end{equation*}
By (\ref{fo1}), there is $C=c\left( \mu ,p\right) $ so that for any $\delta
\in \left( 0,1\right) \,,s\leq \frac{1}{4}\left\vert u-s\right\vert ,$ 
\begin{eqnarray*}
\Delta \left( f;s,t,u\right) &\leq &d\left( f\left( s\right) ,f\left(
0\right) \right) \wedge d\left( f\left( u\right) ,f\left( 0\right) \right)
+d\left( f\left( t\right) ,f\left( 0\right) \right) \\
&\leq &C\left\vert u-s\right\vert ^{\mu }\left( \left[ f\right] _{\mu
}\delta ^{\mu }+\delta ^{-1/p}||f]]_{\mu ,p}\right) .
\end{eqnarray*}%
If $1-u\leq \frac{1}{4}\left\vert u-s\right\vert $, then $u\geq \frac{3}{4}$
and $t\geq 1/4$ because%
\begin{equation*}
1-t=1-\frac{u+s}{2}=1-u+\frac{u-s}{2}\leq \frac{3}{4}\left( u-s\right) \leq 
\frac{3}{4}.
\end{equation*}%
By (\ref{fo2}), there is $C=C\left( \mu ,p\right) $ so that for any $\delta
\in \left( 0,1\right) \,,1-u\leq \frac{1}{4}\left\vert u-s\right\vert ,$ we
have%
\begin{eqnarray*}
\Delta \left( f;s,t,u\right) &\leq &d\left( f\left( s\right) ,f\left(
1\right) \right) \wedge d\left( f\left( u\right) ,f\left( 1\right) \right)
+d\left( f\left( t\right) ,f\left( 1\right) \right) \\
&\leq &C\left\vert u-s\right\vert ^{\mu }\left( \left[ f\right] _{\mu
}\delta ^{\mu }+\delta ^{-1/p}[[f||_{\mu ,p}\right) .
\end{eqnarray*}

Hence%
\begin{equation}
\Delta \left( f;s,t,u\right) \leq C\left( u-s\right) ^{\mu }\left( \left[ f%
\right] _{\mu }\delta ^{\mu }+\delta ^{-1/p}[[f||_{\mu ,p}+\delta
^{-1/p}||f]]_{\mu ,p}\right)  \label{f30}
\end{equation}%
if $\left\vert s\right\vert \leq \frac{1}{4}\left\vert u-s\right\vert $ or $%
\left\vert 1-u\right\vert \leq \frac{1}{4}\left\vert u-s\right\vert .$

According to (\ref{f71}) and (\ref{f30}), there is $C=C\left( \mu ,p\right) $
so that%
\begin{eqnarray*}
&&\Delta \left( f;s,t,u\right) \\
&\leq &C\left( u-s\right) ^{\mu }\left( \left[ f\right] _{\mu }\delta ^{\mu
}++\delta ^{-\frac{3}{p}}\left[ [f\right] ]_{\mu ,p}+\delta
^{-1/p}[[f||_{\mu ,p}+\delta ^{-1/p}||f]]_{\mu ,p}\right) ,
\end{eqnarray*}%
for any $0\leq s<t<u\leq 1,t=\left( s+u\right) /2.$ Hence for all $\delta
\in \left( 0,1\right) ,$%
\begin{equation}
\left[ f\right] _{\symbol{94}\mu }\leq C\left( \left[ f\right] _{\mu }\delta
^{\mu }+\left[ [f\right] ]_{\mu ,p}\delta ^{-3/p}\right)  \label{f52}
\end{equation}%
for some $C=C\left( \mu ,p\right) $. Then by Lemma \ref{l2} and (\ref{f51}),
(\ref{f52}), there is $C_{1}=C_{1}\left( \mu ,p\right) $ so that for all $%
\delta \in \left( 0,1\right) $ we have%
\begin{eqnarray*}
&&\left[ f\right] _{\mu }+|f]_{\mu }+[f|_{\mu } \\
&\leq &C_{1}\left( \left[ f\right] _{\mu }\delta ^{\mu }+\delta ^{-\frac{3}{p%
}}[\left[ f\right] ]_{\mu ,p}+\delta ^{-\frac{1}{p}}||f]]_{\mu ,p}+\delta ^{-%
\frac{1}{p}}[[f||_{\mu ,p}\right) .
\end{eqnarray*}

Choosing $\delta $ so that $\delta ^{\mu }C_{1}\leq 1/2$ we see that for
some $C=C\left( \mu ,p\right) ,$%
\begin{equation*}
\left[ f\right] _{\mu }+|f]_{\mu }+[f|_{\mu }\leq C\left( [\left[ f\right]
]_{\mu ,p}+||f]]_{\mu ,p}+[[f||_{\mu ,p}\right) ,f\in D^{\mu }\left( \left[
0,1\right] ,E\right) .
\end{equation*}

If $E=\mathbf{R}^{k},d\left( x,y\right) =\left\vert x-y\right\vert ,x,y\in 
\mathbf{R}^{k}$, then we can estimate the supremum of $f$. For each $t,$%
\begin{eqnarray*}
\left\vert f\left( t\right) \right\vert &\leq &\left\vert f\left( \tau
\right) -f\left( 0\right) \right\vert +\left\vert f\left( t\right) -f\left(
0\right) \right\vert +\left\vert f\left( \tau \right) \right\vert \\
&\leq &2|f]_{\mu }+\left\vert f\left( \tau \right) \right\vert ,\tau \in 
\left[ 0,1\right] .
\end{eqnarray*}%
Hence%
\begin{equation*}
\left\vert f\left( t\right) \right\vert \leq 2|f]_{\mu
}+\int_{0}^{1}\left\vert f\left( \tau \right) \right\vert d\tau ,
\end{equation*}%
and%
\begin{equation*}
\sup_{0\leq t\leq 1}\left\vert f\left( t\right) \right\vert \leq 2|f]_{\mu
}+\left( \int_{0}^{1}\left\vert f\left( \tau \right) \right\vert ^{p}d\tau
\right) ^{1/p}.
\end{equation*}

The claim of Theorem \ref{t1} follows.


\begin{thebibliography}{99}
\bibitem{b} Billingsley, P., Convergence of Probability Measures, Wiley,
1968.

\bibitem{ch} Chentsov, N.N., Weak convergence of stochastic processes whose
trajectories have no discontinuities of the second kind and the "heuristic"
approach to the Kolmogorov-Smirnov tests, Teor. Veroyatn. Primen., 1, 1956,
pp. 140-144.

\bibitem{c} Cramer, H., On stochastic processes whose trajectories have no
discontinuities of the second kind, Annali di Matematica Pura ed Applicata,
71(1), 1966, pp.85-92.

\bibitem{dpmt} Da Prato, G., Menaldi, J.-L., and Tubaro, L., Some results of
backward Ito formula, Stochastic Analysis and Applications, 25, 2007, pp.
679-703.

\bibitem{dm} Dellacherie, C. and Meyer, P.-A., Probabilties and Potential.A,
North-Holland, 1978.

\bibitem{f} Fernique X., Compactness of distributions of c\`{a}dl\`{a}g
random functions, Lith. Math. J., 1994, 34(3), pp. 231-243.

\bibitem{gs} Gihman, I.I. and Skorokhod, A.V., Theory of Stochastic
Processes, v.1, Springer, 1971.

\bibitem{i} Ibragimov, I. A., Properties of sample functions of stochastic
processes and embedding theorems, Teor. Veroyatn. Primen., 18(3), 1973, pp.
468-480.

\bibitem{k} Kallenberg, O., Foundations of Modern Probability, Springer,
1997.

\bibitem{pz} Peszat, S. and Zabczyk, J., Time regularity of solutions to
linear equations with Levy noise in infinite dimensions, Stochastic
Processes and App., v. 123, 2013, pp. 719-751.

\bibitem{s} Schilling, R.S., Sobolev embedding for stochastic processes,
Expositiones Mathematicae, 2000, v.18, pp. 239-242.
\end{thebibliography}
\end{document}